\documentclass[11pt,reqno]{amsart}
\usepackage{amsmath,amsfonts,amssymb,amsthm,amscd,float}

\usepackage{graphicx}

\usepackage[english]{babel}
\usepackage{hyperref}
\usepackage{subfig}
\usepackage{subfloat}

\usepackage{mathtools}
\mathtoolsset{showonlyrefs}

\newtheorem{thm}{Theorem}[section]
\newtheorem*{at}{Arnold's Theorem}

\newtheorem{lem}[thm]{Lemma}
\newtheorem{prop}[thm]{Proposition}
\theoremstyle{definition}
\newtheorem{defn}[thm]{Definition}
\theoremstyle{remark}

\newtheorem{ex}[thm]{Example}

\numberwithin{equation}{section}

  \usepackage{color}

\usepackage[usenames,dvipsnames]{xcolor}

\begin{document}

\title{Finite type $\xi$-asymptotic lines of plane fields in $\mathbb{R}^3$}%
 
\author{Douglas H.  da Cruz}
 
\author{Ronaldo A. Garcia}
 
\begin{abstract}
We prove that a finite type curve is an $\xi$-asymptotic line (without parabolic points) of a suitable plane field. It is also given
 an explicit  example of a hyperbolic closed finite type $\xi$-asymptotic line. These results obtained here are   generalizations, for plane fields, of the results of V. Arnold \cite[Theorems 2, 3 and 5]{arnold_asy}.
\vskip .3cm
\noindent  {\bf Keywords: } closed asymptotic lines, plane field, parabolic points.
 
 \noindent {\bf  MSC(2000): } {53C12,  37C27, 34C25. } 
\end{abstract}
\maketitle

\section{Introduction}

A regular plane field in $\mathbb{R}^3$ is usually defined by the kernel of   a differential form or an unitary vector field $\xi\colon \mathbb{R}^3\to  \mathbb{R}^3$. In this last case $\xi(p)$ is the normal vector  to the plane at  point $p$.
The classical and germinal work about plane fields in $\mathbb{R}^3$ is \cite{MR1510041}.

The normal curvature of a plane field is defined by (see \cite{MR1749926} and \cite{Euler1760})
\begin{equation}\label{eq:kn}
k_n(p,dr)=-\frac{\langle d\xi(p) ,dr\rangle}{\langle dr, dr \rangle} .
\end{equation}
For  integrable plane fields the  normal curvature is the usual concept of  curves  on  surfaces.

The regular curves $\gamma\colon I\to \mathbb{R}^3$  such that $k_n(\gamma(t), \gamma^\prime(t))=0$ are called {\it $\xi$-asymptotic lines} and the directions $dr$ such $k_n(p,dr)=0$ are called {\it $\xi$-asymptotic directions}.

Recall that asymptotic lines on surfaces are  regular curves $\gamma$ such that $k_n(\gamma(t), \gamma^\prime(t))=0$. Also, asymptotic lines are   the curves $\gamma$ such that the osculating plane of $\gamma$ coincides with the tangent plane of the surface along it, so asymptotic lines are of  extrinsic nature.

The local study, and  singular aspects of asymptotic lines  on   surfaces in $\mathbb{R}^3$, near parabolic points,  is a very classical subject,  see \cite{ MR1725206, MR1634428,  MR2532372}, \cite{faridbook} and references therein.

The  study of closed asymptotic lines of surfaces in $\mathbb{R}^3$ under the viewpoint of qualitative theory of differential equations is more recent,  see \cite{MR1725206, MR1634428, MR2532372}. It is worth to mention that existence of closed asymptotic lines on  the tubes of   \lq\lq T-surfaces\rq\rq \; is still an open problem. See \cite[page 107]{geo2} and \cite{nirenberg}.

Also, it  is not known if there is a surface in $\mathbb{R}^3$ having a cylinder region foliated by closed asymptotic lines (see \cite[page 110]{rozen}). In $\mathbb{S}^3$ all asymptotic lines of the Clifford torus are globally defined and they are  the Villarceau circles.

V. Arnold  in \cite{ arnold_asy} studied the topology  of asymptotic lines  being  curves of type  $(t,t^m,t^n) $   near $t=0$, which are called of finite type.  Also it was shown in  \cite{ arnold_asy} that the projection of a closed asymptotic line of an hyperbolic surface of graph type $(x,y,h(x,y))$ in the horizontal plane $(x,y)$ cannot be  a starlike curve.

The main results of this work are the following.

The Theorem \ref{t5} states that  any
finite type curve is an $\xi$-asymptotic line  
 of a suitable plane field in $\mathbb{R}^3$.

The Theorem \ref{t1} gives an example of a hyperbolic closed finite type $\xi$-asymptotic line of a plane field in $\mathbb{R}^3$.

\section{Preliminaries and Previous Results}

\label{pre}
In this paper, the space $\mathbb{R}^3$ is endowed with the Euclidean norm $|\cdot|=\langle \cdot,\cdot\rangle^{\frac{1}{2}}$.

\begin{defn}[{\cite[Definition 5.15]{zbMATH06351542}}]
A subset $\Omega\subset\mathbb{R}^2$ is called a  \emph{starlike convex set} if there is a point $p\in\Omega$, called the \emph{star point},
such that, for every $q\in\Omega$, the segment $\overline{pq}$ lies in $\Omega$. The boundary of a  starlike convex set
is called a  {\it starlike curve}.
\end{defn}

\begin{thm}[D. Panov, see {\cite{ arnold_asy}}]\label{panov}
The projection of a closed asymptotic line of a surface $z=\varphi(x,y)$ to the plane $\{z=0\}$ cannot
be a starlike curve {\rm(}in particular, this projection cannot be a convex curve{\rm)}.
\end{thm}

\begin{defn}[{\cite{ arnold_asy}}]
A smoothly immersed curve $\gamma:I\to \mathbb R^3$ 
 is said to be of \emph{finite type} at a point $x$, if $\{\gamma^\prime(x), \gamma^{\prime\prime}(x),  \ldots,\gamma^{(k)}(x) \} $  generate
all the tangent space $T_{\gamma(x)}\mathbb R^3$ for some $k\in \mathbb N.$  Here $\gamma^{(k)}(x) $ denotes the  derivative of order $k$ of  $\gamma$.  In a neighborhood of this point, the curve is parametrized locally by $\gamma(x)=(x,a_mx^m+\mathcal{O}^{m+1}(x),b_nx^{n}+\mathcal{O}^{n+1}(x))$,
where $m,n\in\mathbb{N}$, $a_mb_n\ne 0$  and $1<m<n$.

The set 
$\{1,m,n\} $, $(1<m<n)$, of the degrees of $\gamma$ is called the  \emph{symbol} of the point. If $n=m+1$, then $\gamma$ is said to be of \emph{rotating type} at the point.

If a curve is of finite type (resp. rotating type) at every point, then it is called of \emph{finite type curve} (resp. \emph{rotating type curve}).

\end{defn}

A finite type  curve $\gamma$ can be have inflection points, i.e., points where the curvature of $\gamma$ vanishes.

\begin{at}[See {\cite{ arnold_asy}}]\label{arnoldv}
An asymptotic curve of finite type on  a hyperbolic surface is a rotating curve.

Every rotating space curve of finite type is an asymptotic line on a suitable hyperbolic surface.

\end{at}
A new proof of Arnold's Theorem will be given in Appendix \ref{pat}.

\subsection{Plane fields  in \texorpdfstring{$\mathbb{R}^3$}{R3}}

Let $\xi:\mathbb{R}^3\rightarrow\mathbb{R}^3$ be a vector field of class $C^{k}$, where $k\geq3$.

\begin{defn}
A \emph{plane field} $\xi$ in $\mathbb{R}^3$, orthogonal to the vector field $\xi$, is defined by the 1-form $\langle \xi,dr\rangle=0$, where $dr$ is a direction in $\mathbb{R}^3$. See Fig. \ref{pf}.
\end{defn}

\begin{thm}[{\cite[Jacobi Theorem, p.2]{MR1749926}}]\label{jacobi}
There exists a family of surfaces orthogonal to $\xi$ if, and only if, $\langle\xi,curl(\xi)\rangle\equiv0$.
\end{thm}

A plane field $\xi$ is said to be \emph{completely integrable} if $\langle\xi,curl(\xi)\rangle\equiv0$. A surface of the
family of surfaces orthogonal to $\xi$ is called an integral surface.

\begin{figure}
  \includegraphics[angle=90,scale=0.4]{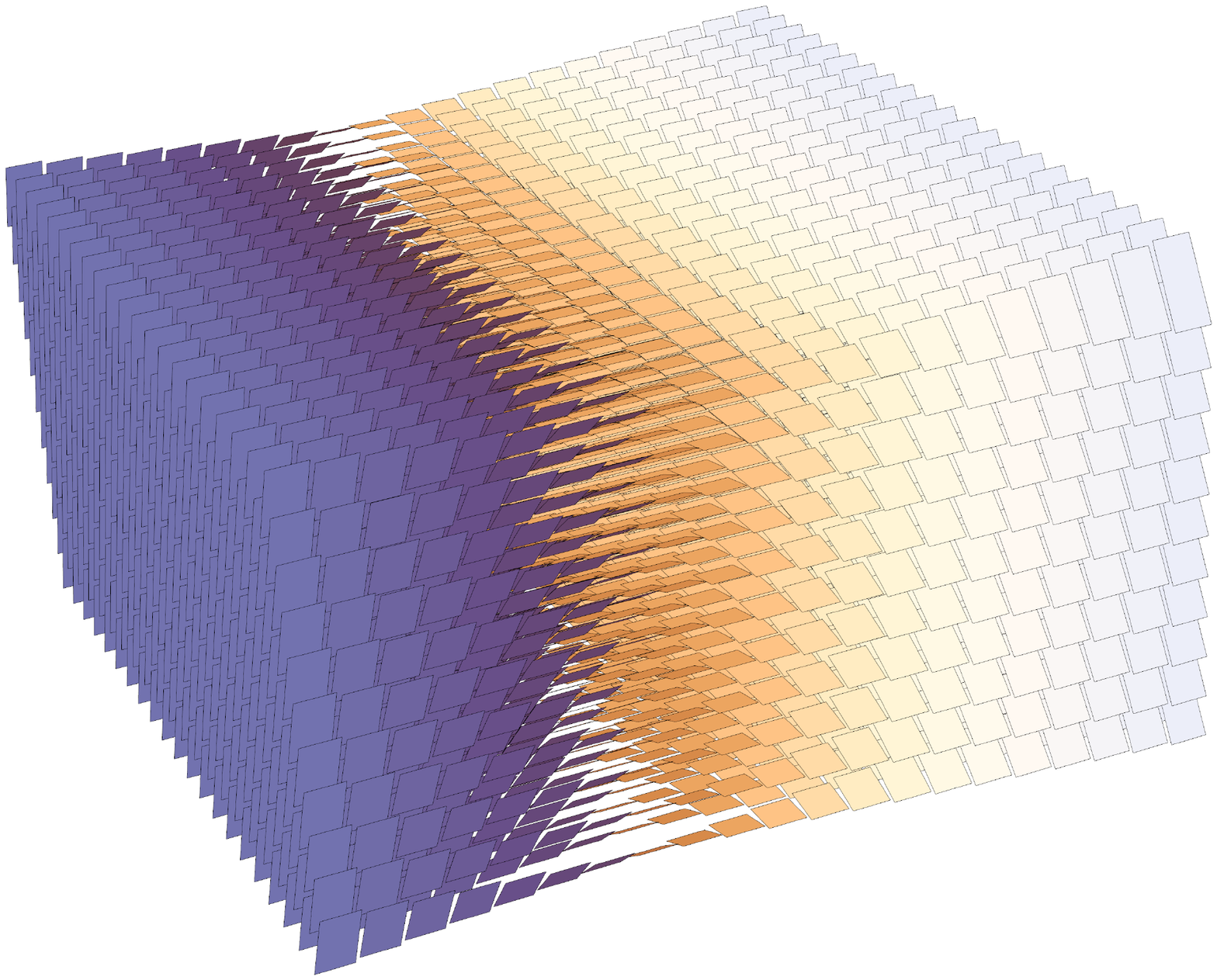}
  \caption{Plane field $\xi$ in $\mathbb{R}^3$ defined by the 1-form $dz-ydx=0$. The vector field
  $\xi$ is given by $\xi(x,y,z)=(-y,0,1)$.}\label{pf}
\end{figure}

\subsection{Normal curvature of a plane field}

\begin{defn}[{\cite[p. 8]{MR1749926}}]
The \emph{normal curvature} $k_{n}$ of a plane field in the direction $dr$ orthogonal to $\xi$   is defined by
\begin{equation}\label{nc}
k_{n}=\frac{\langle d^2r,\xi \rangle}{\langle dr,dr\rangle}=-\frac{\langle dr, d\xi \rangle}{\langle dr,dr\rangle}.
\end{equation}
\end{defn}

This definition agrees with the classical one given by L. Euler, see   \cite{Euler1760}.

In the plane $\pi(p_0,dr)$ generated by $\xi(p_0)$ and $dr$ (direction orthogonal to $\xi(p_0)$) we have a line field  $\ell(p) $ orthogonal to vector $\bar{\xi}(p)\in \pi(p_0,dr)$ obtained projecting $\xi(p)$ in the plane $\pi(p_0,dr)$, with $p\in\pi(p_0,dr)$.  The integral curves $\varphi_p(t)$  of the line field  $\ell  $ are regular curves and $k_n(p_0,dr)$ is the plane curvature of $\varphi_{p_0}(t)$ at $t=0$. See Fig. \ref{fig:euler}.

\begin{figure}[H]
	\begin{center}
		 \includegraphics[scale=0.6]{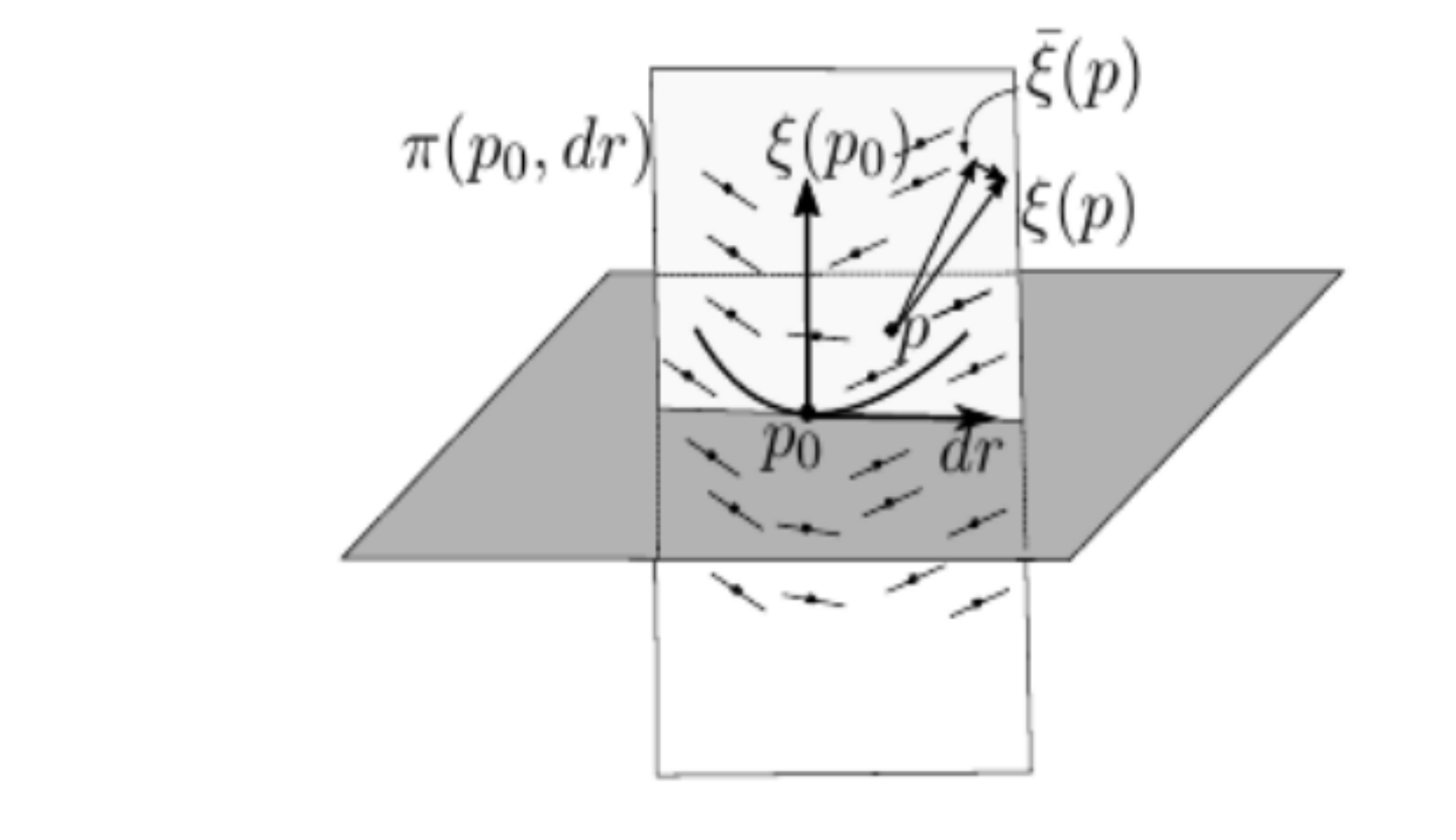}
		\caption {Line field and normal curvature   $k_n(p_0,dr)$. \label{fig:euler} }
	\end{center}
\end{figure}

\subsection{$\xi$-asymptotic lines and parabolic points of a plane field}

The $\xi$-asymptotic directions of a plane field $\xi$ are defined by the following implicit differential equation
\begin{equation}\label{eqla}
\begin{split}
\langle\xi,dr\rangle=0,\ \ \langle d\xi,dr\rangle=0.
\end{split}
\end{equation}
and will referred as the    implicit differential equation  of the $\xi$-asymptotic lines.

A solution $dr$ of  equation \eqref{eqla} is called an  \emph{$\xi$-asymptotic direction}. A curve $\gamma$ in $\mathbb{R}^3$ is an \emph{$\xi$-asymptotic line}  if $\gamma$ is an integral curve of equation \eqref{eqla}. Analogously to the case of asymptotic lines on surfaces, for plane fields the osculating plane of an $\xi-$asymptotic line coincides with the plane of the distribution of planes passing through the point of the curve.  See also \cite[page 29]{MR1796237}.

\begin{defn}
If at a point $r$ there exists two real distinct $\xi$-asymptotic directions (resp. two complex $\xi$-asymptotic directions), then $r$ is called a \emph{hyperbolic point}
(resp. \emph{elliptic point}).

\end{defn}

\begin{defn}
If at $r$ the two $\xi$-asymptotic directions coincide  or  all the directions
are $\xi$-asymptotic directions then $r$ is called a  \emph{parabolic point}.
\end{defn}

\begin{ex}
The circle in $\mathbb{R}^3$ given by $x^2+y^2=1$, $z=0$, is an $\xi$-asymptotic line without parabolic points of the plane field $\xi$ defined by   the orthogonal vector field
$\xi=(\rho,\varrho,\sigma)$, where $\rho=x^2yz+y^3z-x^2y-y^3+xz-2yz+y$, $\varrho=x^3-x^3z-xy^2z+xy^2+2xz+yz-x$ and $\sigma=-x^2-y^2$. The plane field $\xi$ is not completely integrable.
By the Theorem \ref{panov},
 this circle cannot
be an asymptotic line of a regular surface $z=\varphi(x,y)$.
\end{ex}

\begin{prop}\label{prop3}
Given a plane field $\xi$, let $\varphi:\mathbb{R}^3\rightarrow \mathbb{R}$ be a differentiable nonvanishing function. Then a curve $\gamma$ is an $\xi$-asymptotic line     if, and only if,
$\gamma$ is an $\xi$-asymptotic line of the plane field $\widetilde{\xi}$ orthogonal to the vector field $\widetilde{\xi}=\varphi\xi$.
\end{prop}

\begin{proof}
The implicit differential equation  of $\xi$-asymptotic lines of $\widetilde{\xi}$ is  given by
\begin{equation}
\langle\widetilde{\xi},dr\rangle=\varphi\langle\xi,dr\rangle=0, \ \
\langle d\widetilde{\xi}(dr),dr\rangle=d\varphi(dr)\langle\xi,dr\rangle+\varphi \langle d\xi(dr),dr\rangle=0.
\end{equation}
 Then $\gamma$ is an $\xi$-asymptotic line of $\xi$ if, and only if,
$\gamma$ is an $\xi$-asymptotic line of the plane field $\widetilde{\xi}$.
\end{proof}

\subsection{Tubular neighborhood of a integral curve of a plane field}

Let $\xi$ be a plane field orthogonal to a vector field $\xi(x,y,z)$. Then
$d\xi=\xi_{x}dx+\xi_{y}dy+\xi_{z}dz$.
Let $\gamma(x)=(\gamma_{1}(x),\gamma_{2}(x),\gamma_{3}(x))$ be a curve
such that $(\gamma_{1}'(x),\gamma_{2}'(x))\neq(0,0)$ for all $x$. Set $X(x)=\gamma'(x)$, $Y(x)=(\gamma_{2}'(x),-\gamma'(x),0)$,
$Z(x)=(X\wedge Y)(x)$ and $\alpha:\mathbb{R}^3\rightarrow\mathbb{R}^3$,
\begin{equation}\label{tubo}
\alpha(x,y,z)=\gamma(x)+yY(x)+zZ(x).
\end{equation}
The map \eqref{tubo} is a parametrization of a tubular neighborhood of $\gamma$. At this neighborhood, the position point is given by $r=\alpha(x,y,z)$
and then $dr=d\alpha=\alpha_{x}dx+\alpha_{y}dy+\alpha_{z}dz$. It follows that
the implicit differential equation \eqref{eqla} of the $\xi$-asymptotic lines, at this neighborhood, is given by
\begin{equation}\label{eqlaT}
\begin{split}
&\langle \xi,d\alpha\rangle=adx+bdy+cdz=0, \\
&\langle d\xi,d\alpha\rangle=L_{1}dx^2+L_{2}dxdy+L_{3}dy^2+L_{4}dxdz+L_{5}dydz+L_{6}dz^2=0,
\end{split}
\end{equation}
\noindent where,
\begin{equation}
\begin{split}
a=\langle\xi,\alpha_{x}\rangle, \ \ b=\langle\xi,\alpha_{y}\rangle, \ \ c=\langle\xi,\alpha_{z}\rangle,
\end{split}
\end{equation}
\noindent and
\begin{equation}
\begin{split}
&L_{1}=\langle\xi_{x},\alpha_{x}\rangle, \ \ L_{2}=\langle\xi_{x},\alpha_{y}\rangle+\langle\xi_{y},\alpha_{x}\rangle, \ \ L_{3}=\langle\xi_{y},\alpha_{y}\rangle,\\
&L_{4}=\langle\xi_{x},\alpha_{z}\rangle+\langle\xi_{z},\alpha_{x}\rangle, \ \ L_{5}=\langle\xi_{y},\alpha_{z}\rangle+\langle\xi_{z},\alpha_{y}\rangle, \ \ L_{6}=\langle\xi_{z},\alpha_{z}\rangle.
\end{split}
\end{equation}

\begin{prop}\label{propeqlaV}
Let $\gamma(x)=(\gamma_{1}(x),\gamma_{2}(x),\gamma_{3}(x))$ be a curve
such that, for all $x$, $(\gamma_{1}'(x),\gamma_{2}'(x))\neq(0,0)$. Consider the tubular neighborhood \eqref{tubo}.
If $\xi$ is a plane field such that $\frac{a}{c}$ and $\frac{b}{c}$ are well defined in a neighborhood of $\gamma$, where $a,b,c$ are given by \eqref{eqlaT}, then
the implicit differential equation  of the $\xi$-asymptotic lines, in this neighborhood,
is given by
\begin{equation}
\begin{split}\label{eqlaV}
dz=-\left(\frac{a}{c}\right)dx-\left(\frac{b}{c}\right)dy, \quad edx^2+2fdxdy+gdy^2=0,
\end{split}
\end{equation}
\noindent where,
{\small  
\begin{equation}
\begin{split}
e=L_{1}-\frac{aL_{4}}{c}+\frac{a^2L_{6}}{c^2}, \ \ g=L_{3}-\frac{bL_{5}}{c}+\frac{b^2L_{6}}{c^2}, \ \ f=\frac{L_{2}}{2}-\frac{(aL_{5}+bL_{4})}{2c}+\frac{abL_{6}}{2c^2}.
\end{split}
\end{equation}
}
Furthermore, in this neighborhood, the parabolic set of $\xi$ is given by $eg-f^2=0$.
\end{prop}

\begin{proof}
In a neighborhood of $\gamma$, solve the first equation of \eqref{eqlaT} for $dz$ to get the first equation of \eqref{eqlaV}.
Replace this $dz$ in the second equation of \eqref{eqlaT} to get the second equation of \eqref{eqlaV}.

If $eg-f^2<0$ at a point (resp. $eg-f^2>0$), then the equations \eqref{eqlaV} defines two distinct $\xi$-asymptotic directions at this point (resp. two complex $\xi$-asymptotic directions).

If $eg-f^2=0$ at a point, then at it the $\xi$-asymptotic directions coincide or, if $e=g=f=0$, all directions are $\xi$-asymptotic directions.
\end{proof}

\begin{defn}[{\cite[p. 11]{MR1749926}}]
\label{defHK}
Let $\xi$ be a plane field satisfying the assumptions of  Lemma \ref{propeqlaV}. The function defined by $\mathcal{K}=eg-f^2$ is called the
\emph{Gaussian curvature}  of $\xi$.
\end{defn}

\begin{lem}\label{p1}
Let $\gamma(x)=(\gamma_{1}(x),\gamma_{2}(x),\gamma_{3}(x))$ be an $\xi$-asymptotic line of a plane field $\xi$,
such that $(\gamma_{1}'(x),\gamma_{2}'(x))\neq(0,0)$ for all $x$. Consider the tubular neighborhood \eqref{tubo}.
Then, in a neighborhood of $\gamma$, the vector field $\xi$ is given by
{\small  
\begin{equation}\label{arnfxi}
\begin{split}
&\xi(x,y,z)=l_{0}(x)Y(x)+k_{0}(x)Z(x)\\
&+\left(yk_{1}(x)+zl_{1}(x)+\left(\frac{y^2}{2}\right)\widetilde{k}_{1}(x)+yz\widetilde{j}_{1}(x)+\left(\frac{z^2}{2}\right)\widetilde{l}_{1}(x)+A(x,y,z)\right)X(x)\\
&+\left(yk_{2}(x)+zl_{2}(x)+\left(\frac{y^2}{2}\right)\widetilde{k}_{2}(x)+yz\widetilde{j}_{2}(x)+\left(\frac{z^2}{2}\right)\widetilde{l}_{2}(x)+B(x,y,z)\right)Y(x)\\
&+\left(yk_{3}(x)+zl_{3}(x)+\left(\frac{y^2}{2}\right)\widetilde{k}_{3}(x)+yz\widetilde{j}_{3}(x)+\left(\frac{z^2}{3}\right)\widetilde{l}_{3}(x)+C(x,y,z)\right)Z(x),
\end{split}
\end{equation}
}
\noindent where $X(x)=\gamma'(x)$, $Y(x)=(\gamma_{2}'(x),-\gamma_{1}'(x),0)$, $Z(x)=(X\wedge Y)(x)$, $A(x,0,0)=B(x,0,0)=C(x,0,0)=0$ and
\begin{equation}\label{sol2}
k_{0}=\gamma_{1}'\gamma_{2}''-\gamma_{2}'\gamma_{1}'', \ \ l_{0}=(\gamma_{3}'\gamma_{1}''-\gamma_{1}'\gamma_{3}'')\gamma_{1}'+(\gamma_{3}'\gamma_{2}''-\gamma_{2}'\gamma_{3}'')\gamma_{2}'.
\end{equation}
Furthermore, if $\gamma_{1}'(x)\gamma_{2}''(x)-\gamma_{2}'(x)\gamma_{1}''(x)\neq0$ for all $x$, then the implicit differential equation
of
the $\xi$-asymptotic lines  is given by \eqref{eqlaV}.
\end{lem}
\begin{proof}
The expression \eqref{arnfxi} holds,  since $\gamma$ is an integral curve of the plane field defined by $\xi$.
Also, as $\gamma$ is an $\xi$-asymptotic line, $\langle \xi(x),\gamma''(x) \rangle=0$ for all $x$, which gives the expressions \eqref{sol2}.

If $\gamma_{1}'(x)\gamma_{2}''(x)-\gamma_{2}'(x)\gamma_{1}''(x)\neq0$, then $c(x,0,0)\neq0$. The conclusion then follows from   Proposition \ref{propeqlaV}.
\end{proof}

\section{Finite type $\xi$-asymptotic lines of plane fields}

In this section the following result is established.
\begin{thm}\label{t5}
Any finite type curve is an $\xi$-asymptotic line without parabolic points of a suitable plane field.
\end{thm}
\begin{proof}
Let
{\small  
 $\gamma(x)=(\gamma_{1}(x),\gamma_{2}(x),\gamma_{3}(x))=(x,a_{m}x^m+\mathcal{O}^{m+1}(x),a_{n}x^{n}+\mathcal{O}^{n+1}(x))$ 
}
 be a finite type curve.
Consider the tubular neighborhood $\alpha$ given by \eqref{tubo} and  the
 vector field $\xi$ given by \eqref{arnfxi}, with $k_{0}(x)$, $l_{0}(x)$
given by \eqref{sol2}. Then $\gamma$ is a $\xi$-asymptotic line of the plane field orthogonal to $\xi$.

We have that $a(x,0,0)=0$ and
\begin{equation}
\begin{split} b(x,0,0)&=((\gamma_{1}')^2+(\gamma_{2}')^2)[(\gamma_{1}'\gamma_{1}''+\gamma_{2}'\gamma_{2}'')\gamma_{3}'-((\gamma_{1}')^2+(\gamma_{2}')^2)\gamma_{3}''], \\ c(x,0,0)&=((\gamma_{1}')^2+(\gamma_{2}')^2)((\gamma_{1}')^2+(\gamma_{2}')^2+(\gamma_{3}')^2)(\gamma_{1}'\gamma_{2}''-\gamma_{2}'\gamma_{1}'').
\end{split}
\end{equation}
We can factor the term $x^{m-2}$ from $b(x,0,0)$ and $c(x,0,0)$, $b(x,0,0)=x^{m-2}B(x,0,0)$ and $c(x,0,0)=x^{m-2}C(x,0,0)$
where $B(0,0,0)=0$ and $C(0,0,0)=a_{m}m(m-1)$.

Then, after factoring $x^{m-2}$ from the first equation of \eqref{eqlaT}, the equation \eqref{eqlaT} will become
$adx+bdy+cdz=0$, where $a(x,0,0)=0$, $b(0,0,0)=0$ and $c(0,0,0)=a_{m}m(m-1)\neq0$. By Proposition \ref{propeqlaV}, in a neighborhood of $(0,0,0)$, the equation of $\xi$-asymptotic lines are given by \eqref{eqlaV}.

Let $k_{1}(x)$ be defined by
{\small  
\begin{equation}
\begin{split}
k_{1}&=\frac{((\gamma_{1}')^2+(\gamma_{2}')^2)^2[(\gamma_{2}''\gamma_{3}'''-\gamma_{3}''\gamma_{2}''')\gamma_{1}'
+(\gamma_{3}''\gamma_{1}'''-\gamma_{1}''\gamma_{3}''')\gamma_{2}'+(\gamma_{1}''\gamma_{2}'''-\gamma_{2}''\gamma_{1}''')\gamma_{3}']}{((\gamma_{1}')^2+(\gamma_{2}')^2+(\gamma_{3}')^2)(\gamma_{1}'\gamma_{2}''-\gamma_{2}'\gamma_{1}'')}\\
&+\frac{[(\gamma_{1}'\gamma_{1}''+\gamma_{2}'\gamma_{2}'')\gamma_{3}'-((\gamma_{1})^2+(\gamma_{2})^2)\gamma_{3}'']l_{1}
+2(\gamma_{1}'\gamma_{2}''-\gamma_{2}'\gamma_{1}'')}{((\gamma_{1}')^2+(\gamma_{2}')^2+(\gamma_{3}')^2)(\gamma_{1}'\gamma_{2}''-\gamma_{2}'\gamma_{1}'')}.
\end{split}
\end{equation}
}
Then $\mathcal{K}(x,0,0)=-1$.
\end{proof}

\section{Hyperbolic closed finite type $\xi$-asymptotic line}
Examples of hyperbolic asymptotic lines on surfaces are given in \cite{MR1725206, MR1634428,  MR2532372}.

In this section it will be given an example of a hyperbolic closed $\xi$-asymptotic line of finite type for a suitable plane field.

\begin{prop}\label{lac}
Let $\gamma$, $\gamma(x)=(\gamma_{1}(x),\gamma_{2}(x),\gamma_{3}(x))$, be a curve
such that $(\gamma_{1}'(x),\gamma_{2}'(x))\neq(0,0)$, $\gamma_{1}'(x)\gamma_{2}''(x)-\gamma_{2}'(x)\gamma_{1}''(x)\neq0$ for all $x$.
Consider the tubular neighborhood $\alpha$ given by \eqref{tubo} and  the
 vector field $\xi$ given by \eqref{arnfxi}, with $k_{0}(x)$, $l_{0}(x)$
given by \eqref{sol2}. Let $H(x)$ be a nonvanishing function and define $k_{1}(x)$ by
{\small  
\begin{equation}\label{lack1}
\begin{split}
k_{1}&=\frac{((\gamma_{1}')^2+(\gamma_{2}')^2)^2[(\gamma_{2}''\gamma_{3}'''-\gamma_{3}''\gamma_{2}''')\gamma_{1}'
+(\gamma_{3}''\gamma_{1}'''-\gamma_{1}''\gamma_{3}''')\gamma_{2}'+(\gamma_{1}''\gamma_{2}'''-\gamma_{2}''\gamma_{1}''')\gamma_{3}']}{((\gamma_{1}')^2+(\gamma_{2}')^2+(\gamma_{3}')^2)(\gamma_{1}'\gamma_{2}''-\gamma_{2}'\gamma_{1}'')}\\
&+\frac{[(\gamma_{1}'\gamma_{1}''+\gamma_{2}'\gamma_{2}'')\gamma_{3}'-((\gamma_{1})^2+(\gamma_{2})^2)\gamma_{3}'']l_{1}
+2(\gamma_{1}'\gamma_{2}''-\gamma_{2}'\gamma_{1}'')H}{((\gamma_{1}')^2+(\gamma_{2}')^2+(\gamma_{3}')^2)(\gamma_{1}'\gamma_{2}''-\gamma_{2}'\gamma_{1}'')}.
\end{split}
\end{equation}
}
Then, $\gamma$ is a $\xi$-asymptotic line, without parabolic points, of the
plane field orthogonal to the vector field $\xi$.

Furthermore, $\mathcal{K}(x,0,0)=-(H(x))^2$.
 
\end{prop}

\begin{proof}
By direct calculations, we can see that $\gamma$ is a $\xi$-asymptotic line. The implicit differential equation of the $\xi$-asymptotic lines  are given by \eqref{eqlaV} and $e(x,0,0)=0$,
 $f(x,0,0)=H(x)$. Since $e(x,0,0)=0$, then $\mathcal{K}(x,0,0)=-(H(x))^2$ for all $x$.
\end{proof}

\subsection{Poincar\'e map associated to a closed $\xi$-asymptotic line}

Let $\gamma:[0,l]\rightarrow\mathbb{R}^3$, $\gamma(x)=(\gamma_{1}(x),\gamma_{2}(x),\gamma_{3}(x))$, be a closed $\xi$-asymptotic line, without parabolic points, of a plane field $\xi$,
such that $\gamma(0)=\gamma(l)$, $(\gamma_{1}'(x),\gamma_{2}'(x))\neq(0,0)$, $\gamma_{1}'(x)\gamma_{2}''(x)-\gamma_{2}'(x)\gamma_{1}''(x)\neq0$ for all $x$,
and consider the tubular neighborhood $\alpha$ given by \eqref{tubo}.

This means that $\gamma$ is a regular curve having a projection in a plane which is a strictly locally convex curve.

By the Proposition \ref{p1}, $\xi$ is given by \eqref{arnfxi} and the implicit differential equations of the $\xi$-asymptotic lines  is  given by \eqref{eqlaV}.

Let $\Sigma_{x_{0}}=\{(x_{0},y,z)\}$ be a transversal section. Then $\alpha(\Sigma_{x_{0}})$ is the plane spanned by $Y(x_{0})$ and $Z(x_{0})$.
By Lemma \ref{p1}, in a neighborhood of $\gamma$, the $\xi$-asymptotic line passing through
$\alpha(x_{0},y_{0},z_{0})$ intersects $\alpha(\Sigma_{x_{0}})$ again at the point
\[ \alpha(x_{0}+l,y(x_{0}+l,y_{0},z_{0}),z(x_{0}+l,y_{0},z_{0})),\]
\noindent where $(y(x,y_{0},z_{0}),z(x,y_{0},z_{0}))$ is  solution of the following Cauchy problem
\begin{equation}\label{lacp}
\begin{split}
&\frac{dz}{dx}=-\frac{a}{c}-\left(\frac{b}{c}\right)\frac{dy}{dx}= A+B\frac{dy}{dx},\\
&e+2f\frac{dy}{dx}+g\left(\frac{dy}{dx}\right)^2=0,\\
&(y(x_{0},y_{0},z_{0}),z(x_{0},y_{0},z_{0}))=(y_{0},z_{0}).
\end{split}
\end{equation}

The \emph{Poincar\'e map} $\mathcal{P}$, also called {\it   first return map},  associated to $\gamma$ is defined by $\mathcal{P}:\mathcal{U}\subset \Sigma\rightarrow \Sigma$,
$\mathcal{P}(y_{0},z_{0})=(y(l,y_{0},z_{0}),z(l,y_{0},z_{0}))$. See Fig. \ref{fig:poincaremap}.

A closed $\xi$-asymptotic line $\gamma$ is said to be \emph{hyperbolic} if the eigenvalues of $d\mathcal{P}_{(0,0)}$ does not belongs to $\mathbb{S}^1$.
 See  \cite{pm} for the  generic properties of the Poincar\'e map associated to closed orbits of vector fields.

We will denote by $d\mathcal{P}_{(0,0)}$ the matrix of   the first derivative of the Poincar\'e map evaluated at
$(y_{0},z_{0})=(0,0)$.
  	\begin{figure}[H]\
 	\begin{center}
 	 	 \includegraphics[scale=0.7]{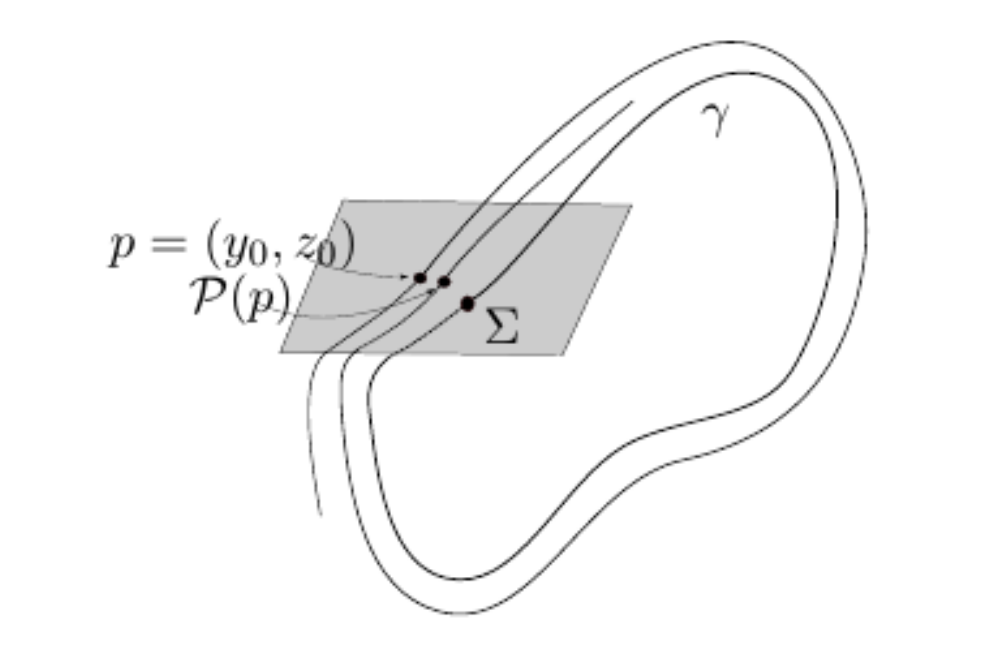}
 		\caption {Poincar\'e return map. \label{fig:poincaremap} }
 	\end{center}
 \end{figure}


\begin{prop}\label{th}
Let $\gamma:[0,l]\rightarrow\mathbb{R}^3$, $\gamma(x)=(\gamma_{1}(x),\gamma_{2}(x),\gamma_{3}(x))$, be a closed $\xi$-asymptotic line, 
 having a projection in a plane which is locally strictly convex curve.
 Let $\mathcal{P}$ be the Poincar\'e map
associated to $\gamma$. Then $d\mathcal{P}_{(0,0)}=\mathcal{Q}(l)$, where $\mathcal{Q}(x)$ is   solution of the following Cauchy problem:
\begin{equation}\label{dp}
\frac{d}{dx}\left(\mathcal{Q}(x)\right)=\mathcal{M}(x)\mathcal{Q}(x), \ \ \mathcal{Q}(0)=\mathcal{I},
\end{equation}
\noindent where $\mathcal{I}$ is the identity matrix, and $\mathcal{M}(x)$, $\mathcal{Q}(x)$ are the matrices given by
{\small
\begin{equation}
\mathcal{M}(x)=
\left(
  \begin{array}{cc}
    -\frac{e_{y}(x,0,0)}{2f(x,0,0)} & -\frac{e_{z}(x,0,0)}{2f(x,0,0)} \\
    \left(A\right)_{y}(x,0,0) & \left(A\right)_{z}(x,0,0) \\
  \end{array}
\right),
\mathcal{Q}(x)=
\left(
  \begin{array}{cc}
    \frac{dy}{dy_{0}}(x,0,0) & \frac{dy}{dz_{0}}(x,0,0) \\
    \frac{dz}{dy_{0}}(x,0,0) & \frac{dz}{dz_{0}}(x,0,0) \\
  \end{array}
\right),
\end{equation}
}
\noindent where $A=-\frac{a}{c}$.
\end{prop}

\begin{proof}
	To fix the notation suppose that $\gamma(0)=\gamma(l)$, $(\gamma_{1}'(x),\gamma_{2}'(x))\neq(0,0)$ and $\gamma_{1}'(x)\gamma_{2}''(x)-\gamma_{2}'(x)\gamma_{1}''(x)\neq0$ for all $x$.
	
Let $(y(x,y_{0},z_{0}),z(x,y_{0},z_{0}))$ be   solution of the Cauchy problem given by equation \eqref{lacp}. Then, at $(y,z)=(0,0)$,
$\frac{dy}{dx}(x,0,0)=\frac{dz}{dx}(x,0,0)=0$.

Differentiating the first equation of \eqref{lacp} with respect to $y_{0}$ (resp. $z_{0}$), it results that:
\begin{equation}\label{eq1}
\begin{split}
\frac{d}{dx}\left(\frac{dz}{dy_{0}}\right)=A_{y}\frac{dy}{dy_{0}}+A_{z}\frac{dz}{dy_{0}}+B\frac{d}{dx}\left(\frac{dy}{dy_{0}}\right)
+\left(B_{y}\frac{dy}{dy_{0}}+B_{z}\frac{dz}{dy_{0}}\right)\frac{dy}{dx},
\end{split}
\end{equation}
\noindent respectively,

\begin{equation}\label{eq2}
\begin{split}
\frac{d}{dx}\left(\frac{dz}{dz_{0}}\right)=A_{y}\frac{dy}{dz_{0}}+A_{z}\frac{dz}{dz_{0}}+B\frac{d}{dx}\left(\frac{dy}{dz_{0}}\right)
+\left(B_{y}\frac{dy}{dz_{0}}+B_{z}\frac{dz}{dz_{0}}\right)\frac{dy}{dx}.
\end{split}
\end{equation}
Differentiating the second equation of \eqref{lacp} with respect to $y_{0}$ (resp. $z_{0}$), it results that:
\begin{equation}\label{eq3}
\begin{split}
&e_{y}\frac{dy}{dy_{0}}+e_{z}\frac{dz}{dy_{0}}+2f\frac{d}{dx}\left(\frac{dy}{dy_{0}}\right)+2\left(f_{y}\frac{dy}{dy_{0}}+f_{z}\frac{dz}{dy_{0}}
+g\frac{d}{dx}\left(\frac{dy}{dy_{0}}\right)\right)\frac{dy}{dx}\\
&+\left(g_{y}\frac{dy}{dy_{0}}+g_{z}\frac{dz}{dy_{0}}\right)\left(\frac{dy}{dx}\right)^2=0,
\end{split}
\end{equation}
\noindent respectively,
\begin{equation}\label{eq4}
\begin{split}
&e_{y}\frac{dy}{dz_{0}}+e_{z}\frac{dz}{dz_{0}}+2f\frac{d}{dx}\left(\frac{dy}{dz_{0}}\right)+2\left(f_{y}\frac{dy}{dz_{0}}+f_{z}\frac{dz}{dz_{0}}
+g\frac{d}{dx}\left(\frac{dy}{dz_{0}}\right)\right)\frac{dy}{dx}\\
&+\left(g_{y}\frac{dy}{dz_{0}}+g_{z}\frac{dz}{dz_{0}}\right)\left(\frac{dy}{dx}\right)^2=0.
\end{split}
\end{equation}
Evaluating \eqref{eq1}, \eqref{eq2}, \eqref{eq3}, \eqref{eq4} at $(y,z)=(0,0)$, it follows that:
\begin{equation}\label{eq5}
\begin{split}
&A_{y}\frac{dy}{dy_{0}}+A_{z}\frac{dz}{dy_{0}}=\frac{d}{dx}\left(\frac{dz}{dy_{0}}\right), \ \ e_{y}\frac{dy}{dy_{0}}+e_{z}\frac{dz}{dy_{0}}+2f\frac{d}{dx}\left(\frac{dy}{dy_{0}}\right)=0,\\
&A_{y}\frac{dy}{dz_{0}}+A_{z}\frac{dz}{dz_{0}}=\frac{d}{dx}\left(\frac{dz}{dz_{0}}\right), \ \ e_{y}\frac{dy}{dz_{0}}+e_{z}\frac{dz}{dz_{0}}+2f\frac{d}{dx}\left(\frac{dy}{dz_{0}}\right)=0.
\end{split}
\end{equation}
Then $\frac{d}{dx}\left(\mathcal{Q}(x)\right)=\mathcal{M}(x)\mathcal{Q}(x)$. Since
$(y(0,y_{0},z_{0}),z(0,y_{0},z_{0}))=(y_{0},z_{0})$, it follows that $\mathcal{Q}(0)=\mathcal{I}$.

Since $\mathcal{P}(y_{0},z_{0})=(y(l,y_{0},z_{0}),z(l,y_{0},z_{0}))$,   the first derivative $d\mathcal{P}_{(0,0)}$ is given by $\mathcal{Q}(l)$.
\end{proof}
\subsection{Example of a hyperbolic closed finite type $\xi$-asymptotic line}

 An explicit example of a hyperbolic closed  $\xi$-asymptotic line is given in the next result.

\begin{thm}\label{t1}
Let $\gamma:[0,2\pi]\rightarrow \mathbb{R}^3$, $\gamma(x)=(sin(x),cos(x),sin^3(x))$, see Fig. \ref{hiper}. Then it  is a hyperbolic finite type $\xi$-asymptotic line of a suitable plane field.
\end{thm}

\begin{proof}
Let $\xi$ be a plane field orthogonal to the vector field $\xi$ given by \eqref{arnfxi}, where $k_{0}(x)$ and $l_{0}(x)$ are given by \eqref{sol2}.
Let $k_{1}(x)$ given by \eqref{lack1}, with $H(x)\equiv1$. Then
\begin{equation}
k_{1}(x)=\frac{3(3cos^2(x)-1)sin(x)l_{1}(x)+24cos^3(x)-18cos(x)-2}{9cos^6(x)-18cos^4(x)+9cos^2(x)+1}.
\end{equation}
By Proposition \ref{lac}, $\gamma$ is a $\xi$-asymptotic line without parabolic points and $\mathcal{K}(x,0,0)=-1$.
Performing the calculations, $e_{z}(x,0,0)=\mathcal{E}(x)+l_{2}(x)$. Solve $e_{z}(x,0,0)=0$ for $l_{2}(x)$.
This vanishes the entry $\left(-\frac{e_{z}(x,0,0)}{2f(x,0,0)}\right)$ of  $\mathcal{M}(x)$ given by Theorem \ref{th}. From \eqref{dp}, it follows that
the eigenvalues of $d\mathcal{P}_{(0,0)}$ are given by
\begin{equation}
{\rm exp}\left(\int_{0}^{2\pi}-\frac{e_{y}(x,0,0)}{2f(x,0,0)}dx\right) \;\;\; {\rm and}\;\;\; {\rm exp}\left(\int_{0}^{2\pi}A_{z}(x,0,0)dx\right).
\end{equation}
 
Set $l_{1}(x)=cos(x)$. Then
{\small  
\begin{equation}
\begin{split}
A_{z}(x,0,0)&=9sin(x)cos^8(x)+54sin(x)cos^6(x)-9cos^6(x)-117sin(x)cos^4(x)\\
&+18cos^4(x)+55cos^2(x)sin(x)-9cos^2(x)-1.
\end{split}
\end{equation}
}
It follows that $\int_{0}^{2\pi}A_{z}(x,0,0)dx=-\frac{25\pi}{8}$. Let  $k_{3}(x)=0$ and  $k_{2}(x)$ a  solution of  the equation
$e_{y}(x,0,0)+2f(x,0,0)=0$. It follows that $$\int_{0}^{2\pi}\left(-\frac{e_{y}(x,0,0)}{2f(x,0,0)}\right)dx=2\pi.$$
\end{proof}

\begin{figure}[H]
\captionsetup[subfigure]{width=.45\linewidth}
     \centering
     \subfloat[][Curve $\gamma(x)=(sin(x),cos(x),sin^3(x))$.]{\includegraphics[width=.4\textwidth]{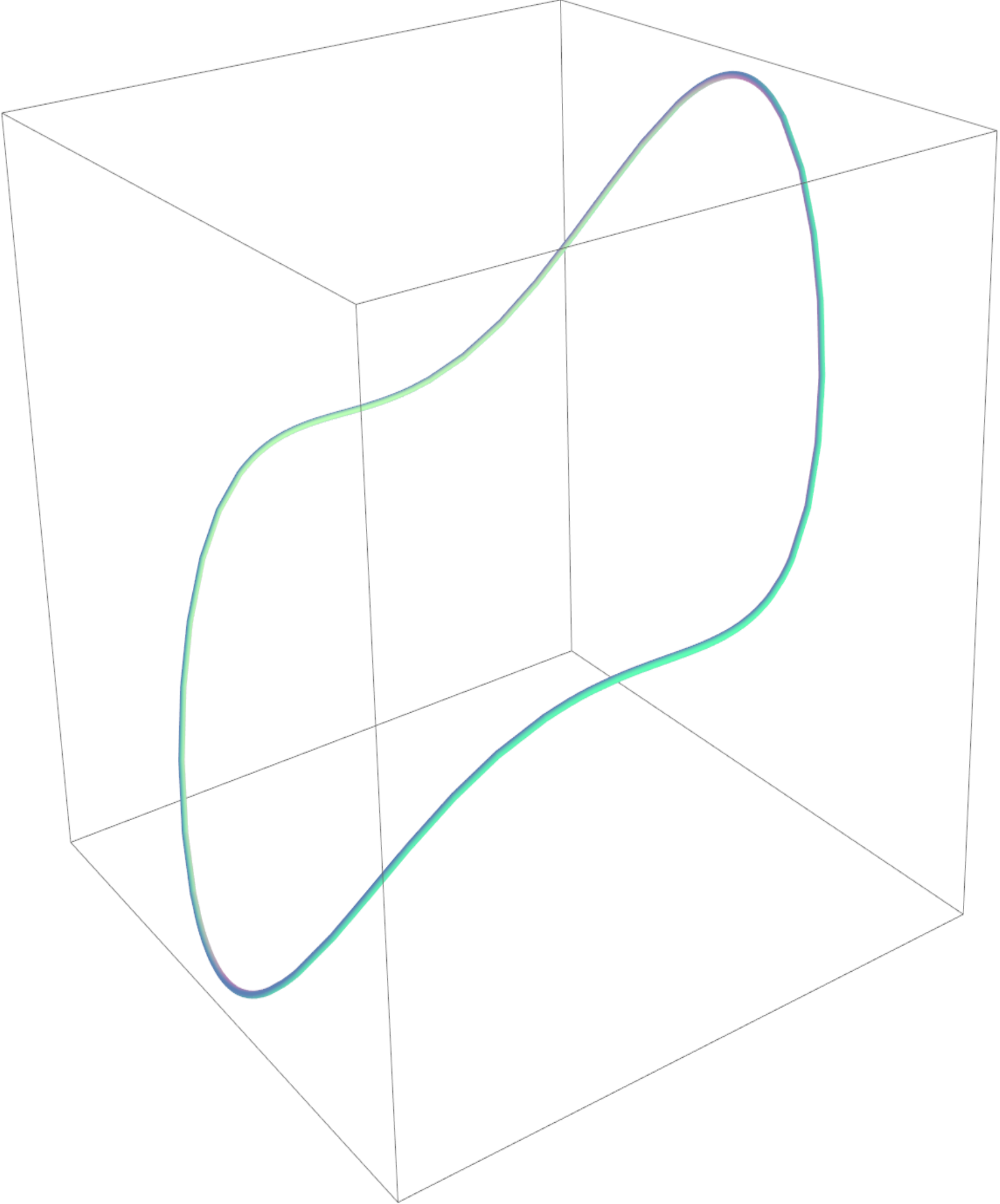}}
     \qquad \qquad \qquad
     \subfloat[][Curve $\gamma(x)=(sin(x),cos(x),sin^3(x))$ on the cylinder $\beta(x,y)=(sin(x),cos(x),y)$.]{\includegraphics[width=.4\linewidth]{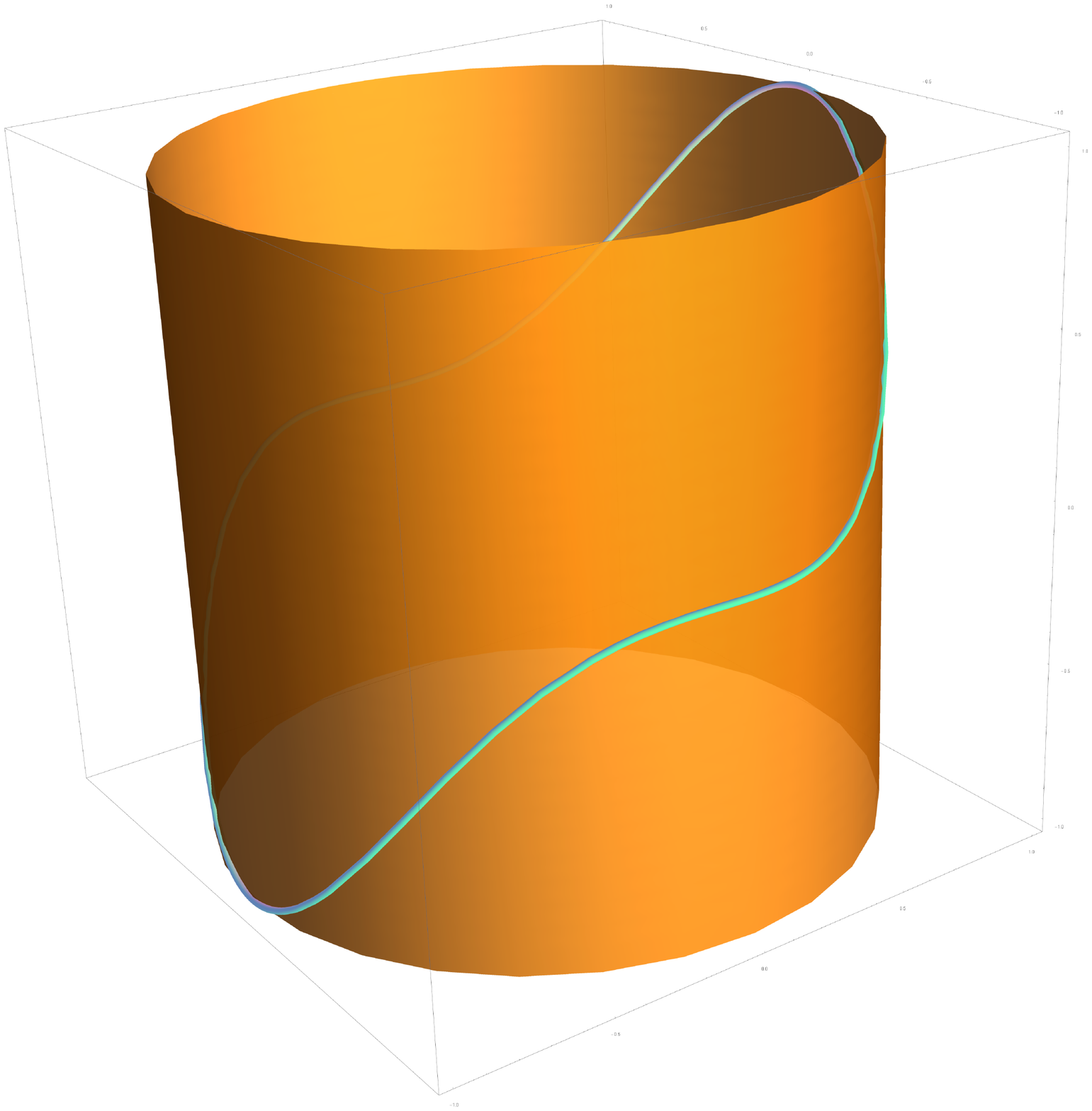}}
     \caption{Finite type curve $\gamma(x)=(sin(x),cos(x),sin^3(x))$.}
     \label{hiper}
\end{figure}
 \section*{Acknowledgments}

The second author is fellow of CNPq. This work was partially supported by Pronex FAPEG/CNPq.


\vskip 1cm

\vskip .2cm {\small \halign {# \hfil & \quad # \hfil \cr
		Authors:& Douglas H. da Cruz and  Ronaldo A. Garcia \cr \cr
		
			Address: &Instituto de Matem\'atica e Estat\'\i stica\cr
		&Universidade Federal de  Goi\'as \cr
		 &Campus Samambaia \cr
		&74690-900 -  Goi\^ania  -G0 - Brasil \cr
		&email:\ ragarcia@ufg.br\cr\cr}}

\appendix
\label{pat}
\section{A new proof of the Arnold's Theorem}

Here it will be given a geometric proof of Arnold's Theorem.

 \begin{proof}
 
Let $\gamma$ be an asymptotic line of finite type $(u,u^m,u^n)$, $n\geq m$. Set $N(u)=(\gamma_{2}'(u),-\gamma_{1}'(u),0)$. Let
\begin{equation}
\alpha(u,v)=\gamma(u)+vN(u)+(k_{1}(u)v+k_{2}(u)v^2+k_{3}(u)v^3+\mathcal{O}^4(v))(\gamma'\wedge N)(u).
\end{equation}
Let
\begin{equation}
N_{\alpha}=\frac{\alpha_{u}\wedge\alpha_{v}}{|\alpha_{u}\wedge\alpha_{v}|}.
\end{equation}
The implicit differential equations of the asymptotic lines of $\alpha$ is given by
\begin{equation}
edu^2+2fdudv+gdv^2=0,
\end{equation}
\noindent where $e=\langle \alpha_{uu},N_{\alpha}\rangle$,
$f=\langle \alpha_{uv},N_{\alpha}\rangle$ and $g=\langle \alpha_{vv},N_{\alpha}\rangle$.

Since $\gamma$ is an asymptotic line of $\alpha$, and parametrized by $v=0$,  we have that $e(u,0) =0$. Then by equation \eqref{lack1} it follows that
\begin{equation}\label{k1}
k_{1}(u)=\frac{[(n-m)m^2u^{2(m-1)}+n-1]nu^{n-m}}{[1+m^2u^{2(m-1)}+n^2u^{2(n-1)}](m-1)m}.
\end{equation}
Direct calculations shows that
\begin{equation}
f(u,0)=\frac{(n-m)(n-1)n(1+m^2u^{2(m-1)})^2u^{n-m-1}}{(m-1)m}.
\end{equation}
It follows that $f(0,0)\neq0$ if, and only if, $n=m+1$.

If $\gamma$ is a rotating space curve of finite type $(u,u^m,u^{m+1})$, $ m\geq 2$, set $N(u)=(\gamma_{2}'(u),-\gamma_{1}'(u),0)$ and let
\begin{equation}\label{eq:la}
\beta(u,v)=\gamma(u)+vN(u)+k_{1}(u)v(\gamma'\wedge N)(u),
\end{equation}
\noindent where $k_{1}(u)$ is given by \eqref{k1} with  $n=m+1$. Therefore,  $e(u,0)=[\beta_u,\beta_v,\beta_{uu}](0,0)= 0$ and $f(0,0)=[\beta_u,\beta_v,\beta_{uv}](0,0)=\frac{m+1}{m-1}\neq 0$. Then $\gamma$ is an asymptotic line,
without parabolic points, of the surface parametrized by $\beta$  in a neighborhood of $(u,v)=(0,0)$.
\end{proof}


\begin{thebibliography}{10}
 	
 	\bibitem{geo2}
 	{\em Geometry. {II}}, volume~29 of {\em Encyclopaedia of Mathematical
 		Sciences}.
 	\newblock Springer-Verlag, Berlin, 1993.
 	\newblock Spaces of constant curvature, A translation of { Geometriya}. II,
 	Akad. Nauk SSSR, Vsesoyuz. Inst. Nauchn. i Tekhn. Inform., Moscow, 1988,
 	Translation by V. Minachin [V. V. Minakhin], Translation edited by \`E. B.
 	Vinberg.
 	
 	\bibitem{MR1749926}
 	Y.~Aminov.
 	\newblock {\em The geometry of vector fields}.
 	\newblock Gordon and Breach Publishers, Amsterdam, 2000.
 	
 	\bibitem{MR1796237}
 	Y.~Aminov.
 	\newblock {\em The geometry of submanifolds}.
 	\newblock Gordon and Breach Science Publishers, Amsterdam, 2001.
 	
 	\bibitem{arnold_asy}
 	V.~I. Arnold.
 	\newblock Topological problems in the theory of asymptotic curves.
 	\newblock {\em Tr. Mat. Inst. Steklova}, 225(Solitony Geom. Topol. na
 	Perekrest.):11--20, 1999.
 	
 	\bibitem{Euler1760}
 	L.~Euler.
 	\newblock Recherches sur la courbure des surfaces.
 	\newblock {\em M{\'e}moires de l'Acad{\'e}mie des Sciences de Berlin},
 	16(119--143):9, 1760.
 	
 	\bibitem{MR1725206}
 	R.~Garcia, C.~Gutierrez, and J.~Sotomayor.
 	\newblock Structural stability of asymptotic lines on surfaces immersed in
 	{$\mathbb{R}^3$}.
 	\newblock {\em Bull. Sci. Math.}, 123(8):599--622, 1999.
 	
 	\bibitem{MR1634428}
 	R.~Garcia and J.~Sotomayor.
 	\newblock Structural stability of parabolic points and periodic asymptotic
 	lines.
 	\newblock {\em Mat. Contemp.}, 12:83--102, 1997.
 	
 	\bibitem{MR2532372}
 	R.~Garcia and J.~Sotomayor.
 	\newblock {\em Differential equations of classical geometry, a qualitative
 		theory}.
 	\newblock Publica\c c\~oes Matem\'aticas do IMPA. [IMPA Mathematical
 	Publications]. Instituto Nacional de Matem\'atica Pura e Aplicada (IMPA), Rio
 	de Janeiro, 2009.
 	
 	\bibitem{faridbook}
 	S.~Izumiya, M.~d.~C. Romero~Fuster, M.~A.~S. Ruas, and F.~Tari.
 	\newblock {\em Differential geometry from a singularity theory viewpoint}.
 	\newblock World Scientific Publishing Co. Pte. Ltd., Hackensack, NJ, 2016.
 	
 	\bibitem{zbMATH06351542}
 	S.~G. {Krantz}.
 	\newblock {\em {Convex analysis.}}
 	\newblock CRC Press, 2015.
 	
 	\bibitem{nirenberg}
 	L.~Nirenberg.
 	\newblock Rigidity of a class of closed surfaces.
 	\newblock In {\em Nonlinear {P}roblems ({P}roc. {S}ympos., {M}adison, {W}is.,
 		1962)}, pages 177--193. Univ. of Wisconsin Press, Madison, Wis., 1963.
 	
 	\bibitem{pm}
 	J.~Palis, Jr. and W.~de~Melo.
 	\newblock {\em Geometric theory of dynamical systems}.
 	\newblock Springer-Verlag, New York-Berlin, 1982.
 	\newblock An introduction, Translated from the Portuguese by A. K. Manning.
 	
 	\bibitem{rozen}
 	E.~R. Rozendorn.
 	\newblock Surfaces of negative curvature.
 	\newblock In {\em Current problems in mathematics. {F}undamental directions,
 		{V}ol. 48 ({R}ussian)}, Itogi Nauki i Tekhniki, pages 98--195. Akad. Nauk
 	SSSR, Vsesoyuz. Inst. Nauchn. i Tekhn. Inform., Moscow, 1989.
 	
 	\bibitem{MR1510041}
 	A.~Voss.
 	\newblock Geometrische {I}nterpretation der {D}ifferentialgleichung
 	{$Pdx+Qdy+Rdz=0$}.
 	\newblock {\em Math. Ann.}, 16(4):556--559, 1880.
 	
 \end{thebibliography}
\end{document}